\documentclass[12pt,twoside]{article}
\usepackage{amssymb,amsmath,amsthm,xcolor}
\usepackage[noadjust]{cite}

\usepackage[german,english]{babel}
\usepackage{geometry}
\newtheorem{theorem}{Theorem}

\newtheorem{lemma}[theorem]{Lemma}

\theoremstyle{definition}

\newtheorem{remarks}[theorem]{Remarks}

\numberwithin{equation}{section}

\DeclareSymbolFont{bbold}{U}{bbold}{m}{n}
\DeclareSymbolFontAlphabet{\mathbbold}{bbold}

\newcommand{\R}{\mathbb{R}}

\renewcommand{\Re}{\operatorname{Re}}

\renewcommand\le{\leqslant}
\renewcommand\ge{\geqslant}
\renewcommand{\|}{|\!|}

\newcommand{\set}[2]{\bigl\{#1{;}\breakOK\;#2\bigr\}}

\newcommand\operator[1]{\expandafter\newcommand\csname#1\endcsname{\operatorname{#1}}}
\operator{vol}
\operator{spt}
\operator{dist}

\newcommand{\oop}{{\!\raisebox{-0.1ex}{$\scriptstyle1/p$}}}

\newcommand{\breakOK}{\penalty0}

\makeatletter
\renewcommand{\.}{.\nobreak\@ifnextchar1{\hskip0.18em plus0.05em minus
0.02em}{\hskip0.22em plus0.06em minus 0.04em}}
\makeatother

\begin{document}

\medmuskip=4mu plus 2mu minus 3mu
\thickmuskip=5mu plus 3mu minus 1mu
\abovedisplayshortskip=0pt plus 3pt minus 3pt
\belowdisplayshortskip=\belowdisplayskip
\newcommand\smallads{\vadjust{\vskip-5pt plus-1pt minus-1pt}}
\newcommand\smallbds{\vskip-1\lastskip\vskip5pt plus3pt minus2pt\noindent}

\title{\Large Stability of uniformly eventually positive $C_0$-semigroups \\
on $L_p$-spaces}
\author{Hendrik Vogt%
\footnote{Fachbereich 3 -- Mathematik, Universit{\"a}t Bremen, 28359 Bremen, Germany,
+49\,421\,218-63702,
{\tt hendrik.vo\rlap{\textcolor{white}{hugo@egon}}gt@uni-\rlap{\textcolor{white}{hannover}}bremen.de}}}
\date{}

\maketitle

\begin{abstract}
We give a short and elementary proof of the theorem of Lutz Weis
that the growth bound of a positive $C_0$-semigroup on $L_p(\mu)$
equals the spectral bound of its generator.
In addition, we generalise the result to the case of uniformly eventually positive semigroups.

\vspace{8pt}


\vspace{2pt}

\end{abstract}

\bigskip

Throughout let $(\Omega,\mu)$ be a measure space, and let $1<p<\infty$;
we do not assume $\mu$ to be $\sigma$-finite.
Let $T$ be a $C_0$-semi\-group on $L_p(\mu)$.
A basic question of stability theory is when the growth bound
\[
  \omega_0(T)
  = \inf\set{\omega\in\R}{\exists M\ge1\,\forall t\ge0\colon \|T(t)\| \le Me^{\omega t}}
\]
of the semigroup~$T$ coincides with the spectral bound
\[
  s(A) = \sup\set{\Re\lambda}{\lambda\in\sigma(A)}
\]
of its generator~$A$.
In \cite{wei95}, Weis proved that this is the case if the semigroup~$T$ is positive;
in \cite{wei98} he gave a shorter proof of the same fact.

The aim of this paper is to give an even shorter and more elementary proof.
Also, we show that the result generalises to the more general
case of uniformly eventually positive semigroups.
As in \cite[Def\.5.1]{dgk16} we say that $T$ is \emph{uniformly eventually positive}
if there exists $t_0\ge0$ such that $T(t) \ge 0$ for all $t\ge t_0$.
There are several important examples of operators that generate
uniformly eventually positive semigroups, such as certain Dirichlet-to-Neumann operators
and bi-Laplace operators, see \cite[Sec\.4]{dagl18}.

We denote by $L_p(\mu)_+$ the set of positive functions in $L_p(\mu)$.

\begin{lemma}\label{lemma}
Let $I\subseteq\R$ be a compact interval, and let
$g\in C(I;L_p(\mu)_+)$. Then
\[
  \left( \int_I g(t)^p\,dt \right)^\oop
  = \sup\set{\textstyle\int_I h(t)g(t)\,dt}{h\in L_{p'}(I;\R)_+,\ \|h\|_{p'}\le1},
\]
where the integrals are Bochner integrals
and the supremum is taken in the Banach lattice $L_p(\mu)$.
\end{lemma}
\begin{proof}
Let $D\subseteq I$ be a countable dense subset.
Then the set $\Omega_0 := \bigcup_{t\in D} \set{x\in\Omega}{g(t)(x)\ne\nobreak0}$ is $\sigma$-finite,
and $g(t)|_{\Omega\setminus\Omega_0}=0$ a.e.\ for all $t\in I$.
Now by Fubini's theorem there exists a function $\tilde g\in L_p(I\times\Omega_0)_+$ such that
$g(t)|_{\Omega_0}=\tilde g(t,\cdot)$ for a.e.\ $t\in I$
(see \cite[Prop\.1.2.24]{hnvw16}).
Then $\bigl(\int_I g(t)^p\,dt\bigr)(x) = \int_I \tilde g(t,x)^p\,dt$
and $\bigl(\int_I h(t)g(t)\,dt\bigr)(x) = \int_I h(t)\tilde g(t,x)\,dt$
for a.e.\ $x\in\Omega_0$.
Thus the assertion follows from the separability
of $L_{p'}(I;\R)$ and the validity of the asserted identity for scalar-valued $g$.
\end{proof}

\begin{theorem}\label{theorem}
\newcommand{\T}{\hskip0.08em T\hskip-0.08em}
Let $\T$ be a uniformly eventually positive $C_0$-semi\-group on $L_p(\mu)$.
Then the growth bound of\/ $\T$ equals the spectral bound of its generator.
\end{theorem}
\begin{proof}
Let $A$ be the generator of $T$, and assume that $s(A)<0$.
We show $\omega_0(T)<0$, using Datko's theorem; then the assertion follows
since $s(A)\le\omega_0(T)$ is always true.

Let $t_0\ge0$ be such that $T(t) \ge 0$ for all $t\ge t_0$.
Let $f\in L_p(\mu)_+$ and $t_1>2t_0+1$.
Let $h\in L_{p'}(2t_0,\infty)_+$ satisfy $h|_{(2t_0,2t_0+1)} = 0$,\,
$h|_{(t_1,\infty)} = 0$ and $\|h\|_{p'} \le 1$. Then
\[
  \int_{2t_0+1}^{t_1} h(t)T(t)f\,dt
  = \int_{2t_0+1}^\infty \int_{t-t_0-1}^{t-t_0} ds\, h(t)T(t)f\,dt
  = \int_{t_0}^\infty \!\! \int_{s+t_0}^{s+t_0+1} h(t)T(t)f\,dt\,ds
\]
by Fubini's theorem. Lemma~\ref{lemma} shows that
\[
  \int_{t_0}^{t_0+1} h(t+s)T(t)f\,dt
  \le \left( \int_{t_0}^{t_0+1} \bigl(T(t)f\bigr)^p dt \right)^\oop
  =: g \in L_p(\mu),
\]
so
\[
  \int_{s+t_0}^{s+t_0+1} h(t)T(t)f\,dt
  = T(s) \int_{t_0}^{t_0+1} h(t+s)T(t)f\,dt
  \le T(s) g
\]
for all $s\ge t_0$, by the positivity of $T(s)$. It follows that
\[
  \int_{2t_0+1}^{t_1} h(t)T(t)f\,dt
  \le \int_{t_0}^\infty T(s)g\,ds =: g_1 \in L_p(\mu),
\]
where the latter integral is convergent as an improper Riemann integral
by our assumption $s(A)<0$; see \cite[Prop\.7.1]{dgk16}.
Using Lemma~\ref{lemma} again we conclude that
$
  \int_{2t_0+1}^{t_1} \bigl(T(t)f\bigr)\rule{0pt}{1.8ex}^p dt \le g_1^p
$
for all $t_1>2t_0+1$ and hence
\[
  \int_{2t_0+1}^\infty \|T(t)f\|_p^p\,dt \le \|g_1\mkern-1.5mu\|_p^p < \infty.
\]
Now Datko's theorem (see \cite[Thm\.V\kern-0.05em.1.8]{enna00})
implies that $\omega_0(T)<0$.
\end{proof}

\begin{remarks}
(a) It remains an open question whether Theorem~\ref{theorem} is also true
if $T$ is merely an \emph{individually} eventually positive $C_0$-semi\-group,
i.e., if for all $f\in L_p(\mu)_+$ there exists $t_0\ge0$ such that $T(t)f\ge0$
for all $t\ge t_0$.
Note that our proof breaks down in this case because it relies on the positivity
of the operator~$T(s)$ for $s\ge t_0$.
If $p=1$ or $p=2$, then the answer to the above question is positive; see \cite[Thm\.7.8]{dgk16}.

(b) In \cite[Thm\.3.8]{rove18}, sharp estimates are given for the asymptotic behaviour
of positive $C_0$-semigroups on $L_p(\mu)$ in terms of resolvent bounds.
This result can be viewed as a quantitative version of Theorem~\ref{theorem}
for the case of positive semigroups.
\end{remarks}

\noindent
\textbf{Acknowledgment.}
I would like to thank Jochen Gl\"uck for valuable discussions.

\end{document}